\documentclass{article}
\usepackage{graphicx, hank, quiver, cleveref} 
\usepackage[style=alphabetic]{biblatex}
\title{Flat Vector Bundles on Very General Curves and the Codimension of Non-Abelian Hodge Loci}
\author{Nathan H. Morris}
\date{\today}

\newcommand{\ssbound}{\sqrt{g-1} + 1}

\newcommand{\codimlowerbound}{g - (r-\ell)^2 - \ell}
\begin{document}

\maketitle
\begin{abstract}
    We bound the codimension of components of the nonabelian Hodge loci in the relative de Rham moduli space over $\shm_{g,n}$ in terms of the rank and level of a complex variation of Hodge structure. If the rank is $r$ and the level is $\ell$, then the codimension of such a component is tightly controlled by $r, \ell$, and $g$. The key input is a generalization of a bound on the rank of flat vector bundles by Landesman and Litt, which we apply to the isomonodromy foliation on the relative de Rham space. 
\end{abstract}
\section{Introduction}
    Let $\shm_{g,n}$ be the moduli space of $n-$pointed curves, and let $\shc$ be the universal curve. When $X$ is a compact curve with a reduced divisor $D = x_1 + \dots + x_n$, Simpson (\cite{PMIHES_1994__80__5_0}) has constructed three relative moduli spaces over this family: $\shm_B (\shc/\shm_{g,n}, \GL_r)$, the relative Betti moduli space of semisimple local systems; $\shm_{dR}(\shc/\shm_{g,n},\GL_r),$, the relative de Rham moduli space of semistable rank $r$ flat vector bundles; and $\shm_{Dol}(\shc/\shm_{g,n}, \GL_r)$, the relative moduli space of polystable Higgs bundles $(E, \theta)$ with vanishing rational Chern class. The de Rham and Dolbeault moduli spaces are algebraic varieties, while the Betti moduli space has the structure of a local system of affine varieties. The fact that the Betti and de Rham spaces are complex-analytically isomorphic is the Riemann-Hilbert correspondence, and that the de Rham and Dolbeault space are real analytically isomorphic is the \textit{nonabelian Hodge correspondence.} 

    Within $\shm_{B}(\shc/\shm_{g,n}, \GL_r)$ we have the locus $\shm_B(\shc/\shm_{g,n}, \GL_r(\ZZ))$ of local systems which have integral monodromy representations. The Dolbeault moduli space $\shm_{Dol}(\shc/\shm_{g,n}, \GL_r)$ has a natural $\GG_m$ action given by scaling the Higgs field $\theta$, and the fixed points of this action are those Higgs bundles underlying a polarizable complex variation of Hodge structure. Denote this fixed point set by $\mathrm{Hdg}$.  Under the nonabelian Hodge correspondence, we may intersect the images of $\mathrm{Hdg}$ and $\shm_B(\shc/\shm_{g,n}, \GL_r(\ZZ))$ in $\shm_{dR} (\shc/\shm_{g,n}, \GL_r)$, obtaining the nonabelian Hodge locus $\NHL(g,n,r)$, the set of flat vector bundles underlying a $\ZZ$-polarizable variation of Hodge structure. 

    The nonabelian Hodge locus is an object of great interest, but not much is known about it. Simpson (\autocite[Theorem 12.1]{simpson1996hodgefiltrationnonabeliancohomology}) has shown that it is a union of analytic subvarieties, and conjectured that it is a union of \textit{algebraic} subvarieties. This has been proven by \cite{engel2023nonabelianhodgelocusi} in the particular case of representations whose integer points are cocompact the Zariski closure of their images. Even so, much of the geometry of these is mysterious.

    In this note, we make one advance in understanding the geometry of the nonabelian Hodge loci by bounding the codimension of its components. Recall that an $n$-pointed complex curve of genus $g$ is \textit{hyperbolic} if $2g-2+n >0$.  If $\VV$ is a variation of Hodge structure with Hodge filtration $F^\bullet$, we define the \textit{level} to be $b-a$, where $[a,b]$ is the smallest interval where $F^p \VV/ F^{p-1}\VV = 0$ for all $p \notin [a,b]$.
    
    The relative de Rham moduli space has a natural foliation given by isomonodromic deformation (see Section 4 for more.) This is obtained by deforming a pointed curve $(C,D)$ together with a flat vector bundle $(E, \nabla)$ such that the corresponding monodromy representation remains constant. We will be interested in flat vector bundles on $C$ with quasi-unipotent monodromy at the points of $D$, which naturally have the structure of a logarithmic bundle. We denote the corresponding moduli space by $\shm_{dR}(\shc/\shm_{g,n}, \shd, r)$. We prove:

    \begin{theorem}\label{nhldimension}
        Consider the relative de Rham moduli space of logarithmic vector bundles of rank $r$ over $\shm_{g,n}$ with regular singularities at the $n$, marked points. Let $Z \subset \Hdg$ be a connected component of the locus of logarithmic Higgs bundles whose very general point underlies a polarizable complex variation of Hodge structure of level $\ell$, and denote its image in $\shm_{dR}(\shc/\shm_{g,n}, \shd,r)$ by $Z^{dR}$. Let $\shl \subset \shm_{dR}(\shc/\shm_{g,n}, \shd, r)$ be a leaf of the isomonodromy foliation. Then the codimension $\delta$ of $\shl \cap Z$ satisfies 
       $$ \delta \geq \codimlowerbound$$
    \end{theorem}

    This implies the same codimension bound for components of $\NHL(\shc/\shm_{g,n}, \shd, r)$ since intersecting with the locus of flat bundles whose corresponding local system has a $\ZZ$ structure can only increase the codimension. This restricts the possible codimension of such loci, especially when $r$ and $\ell$ are relatively small compared to $g$. 

    The proof of \Cref{nhldimension} will follow more or less immediately from the following theorem, which is a generalization of a result of Landesman and Litt (cf. \cite[Theorem 1.3.4a and Corollary 1.3.6]{landesman2025geometriclocalsystemsgeneral}). Their result concerns flat vector bundles whose first-order deformations are not semistable. For a variation of Hodge structures, semistability coincides with having a flat Hodge filtration, and conversely, a longer Hodge filtration means the flat bundle admits many subsheaves. Following this lead, we generalize their result by examining isomonodromic deformations of the Hodge filtration. Our result has the advantage of incorporating the level of the filtration into the overall rank bound.
    
    Instead of proving a result just for complex variations of Hodge structure, we work in the more general setting a parabolic bundle with irreducible connection, see Section 2.1 for background on parabolic bundles. Any such bundle has a Simpson filtration, which in the case of a $\CC$-VHS is just the Hodge filtration, see Section 2.2.3. The Simpson filtration has a numerical invariant called the level, generalizing the weight of the Hodge filtration, and it is this invariant which we incorporate into our rank bound. 

    \begin{theorem}\label{rankbound}
        Let $(C,D)$ be a punctured curve equipped with a flat vector bundle $(E,\nabla)$ with irreducible monodromy and regular singularities at $D$. Suppose that the eigenvalues of the residues of $\nabla$ at the points of $D$ do not differ by consecutive integers. Let $E_\star$ be a parabolic structure on $E$ induced by $\nabla$. Suppose $(E_\star, \nabla)$ admits an isomonodromic deformation $(E_\star', \nabla')$ such that the Simpson filtration extends to a codimension $\delta$ locus of $\shm_{g,n}$. If the Simpson filtration $F^\bullet$ of $E_\star'$ has level $\ell$ and the associated graded Higgs bundle $\gr_{F^\bullet} E_\star$ is stable, then 
        $\rk E_\star \geq \sqrt{g -\delta -\ell} + \ell$.
    \end{theorem}

    The corresponding statement for a $\CC$-VHS follows by taking $E_\star$ above to be its Deligne canonical extension (see Section 2.2) with the parabolic structure induced by the connection, which, due to our condition of the eigenvalues of the residues of $\nabla$, will be unique. We observe that when $\ell \geq 1$ (i.e. $E$ is not semistable), our result nearly recovers Theorem 1.3.4 of \Cite{landesman2025geometriclocalsystemsgeneral}. 
    \begin{corollary}\label{semistability}
        With the notation as in \Cref{rankbound}, if the Simpson filtration $F^\bullet$ extends to the universal isomonodromic deformation of $E_\star$, and  $\rk E_\star < \ssbound$, then $E_\star$ is semistable.
    \end{corollary}
    	\begin{proof}
    		Having a nontrivial Simpson filtration extend to the universal isomonodromic deformation implies that $E_\star$ is not semistable and that the term $\delta$ in $\Cref{rankbound}$ is $0$. Simplifying the bound, we obtain
    		$$ \rk E_\star = r > \sqrt{g - \ell} + \ell.$$
    		We note that the right hand side is increasing as a function of $\ell$ when $\ell \in [1, g-1]$, so it is minimized by taking $\ell = 1$. In this case we arrive at $r > \sqrt{g-1} + 1$. On the other hand, if $\ell > 2g-1$, then $r > 2g -1$. Taking the minimum of these bounds, we see that the filtration  is forced to be trivial when $r \leq \sqrt{g-1} + 1$. 
    	\end{proof}
	
	An important setting for these results is in the case of a variation of Hodge structure, in which case the Hodge filtration is a Simpson filtration. In this case, one only has to assume irreducibility of the connection in the hypotheses of \Cref{rankbound}, as the stability of the graded Higgs bundle follows automatically \cite{simpson2008iterateddestabilizingmodificationsvector}[5.2]. If the isomonodromic deformation of $E_\star$ happens to underlie a variation of Hodge structure, then \Cref{semistability} may be upgraded to recover \cite{landesman2025geometriclocalsystemsgeneral}[Theorem 1.2.13] as well:
	
	\begin{corollary}\label{unitarymonodromy}
		With the notation as in \Cref{rankbound}, if an isomonodromic deformation of $(E_\star, \nabla)$ underlies a variation of Hodge structure and $\rk E_\star < \ssbound$, then $(E_\star, \nabla)$ has unitary monodromy. 
	\end{corollary}
	
	In fact, our methods prove another similar result for those representations which underlie a variation of Hodge structure. Recall that for a representation $\rho: \pi_1(X,x) \to \GL_r(\CC)$ of the fundamental group of a complex variety $X$, the algebraic monodromy group $G$ is the connected component of the identity of the Zariski closure of $\img \rho$. Its Lie algebra, $\mathfrak{g}$, has the natural structure of a flat vector bundle on $X$, and in the case of a representation which underlies a variation of Hodge structure, $\mathfrak{g}$ is actually a sub-variation of Hodge structure. By similar methods to the proof of \Cref{rankbound}, we are able to provide a bound on the rank of $\mathfrak{g}$ under the same hypotheses.
	
	\begin{theorem}\label{liealgbound}
		Let $(C,D)$ be a punctured curve equipped with a flat vector bundle $(E,\nabla)$ which underlies a variation of Hodge structure. Suppose that the eigenvalues of the residues of $\nabla$ at the points of $D$ do not differ by consecutive integers, and let $(\bar{E}_\star, \nabla)$ denote the Deligne canonical extension of $(E, \nabla)$ to $C \cup D$. Let $\mathfrak{g}$ be the Lie algebra of the algebraic monodromy group of the representation associated to $(E_\star, \nabla)$. Suppose that $(\bar{E}_\star,\nabla)$ admits an isomonodromic deformation to the germ of a codimension $\delta$ locus. Then $\rank \gr_F^1 \mathfrak{g} > g-\delta$. If the level of the Hodge filtration of $\mathfrak{g}$ is $2k+1$, then $\rank \mathfrak{g} > g + 2k - \delta$.
	\end{theorem}

    \Cref{rankbound} follows the same strategy of \cite{landesman2025geometriclocalsystemsgeneral}, using a cohomological interpretation of isomonodromic deformations to extract our numerics. Instead of analyzing deformations of the Harder-Narasimhan filtration, we use the Simpson filtration. This incorporates the level of the filtration into the calculations, but at the cost of semistability of the associated graded pieces. While we no longer have the semistability of the quotient pieces of the Harder-Narasimhan filtration to leverage, the semistability of the graded Higgs bundle of $E_\star$ is an adequate substitute that leads to a more streamlined analysis.
    
    \subsubsection{Outline} In \S 1, we collect the facts about parabolic bundles that we will need to prove our results. Next in \S 2 we give an overview of the non-abelian Hodge correspondence, and in \S 3 we review the theory behind isomonodromic deformations. These sections are preliminary, and the proofs of the main results are all contained in \S 4. 
    
    \subsubsection{Acknowledgements}
    I would like to thank my advisor, Benjamin Bakker, for innumerable helpful conversations related to this work. I would also like to thank Daniel Litt for some helpful comments on an earlier draft of this article. 
\section{Background on Parabolic Bundles}
        In this section, we recall the basic definitions concerning parabolic bundles that we will need to prove our results. We establish no new facts about parabolic bundles, so we review this material primarily to set notation. \cite[Section 2]{landesman2025geometriclocalsystemsgeneral} contains all the facts about parabolic bundles which we need, and we follow their exposition.

        For the following definition, $C$ will be a smooth hyperbolic curve and $D = x_1, + \dots + x_n$ will be a reduced divisor. Write $X = C \cup D$ for the resulting compact curve. 
        \begin{definition}
            A parabolic bundle on $(C,D)$ is a vector bundle $E$ with, for each $x_i \in D$, a decreasing filtration $E_{x_i} = E_i^1 \supset E_i^2 \supset \dots \supset E_i^{n_i + 1} = 0$, together with an increasing sequence of real numbers, called weights, $0 \leq \alpha_j^1 < \alpha_j^2 < \dots < \alpha_i^{n_j} < 1$, again for each $x_i \in D$. We use $E_\star$ to denote the triple $(E, \{E_i\}_{x_i \in D}, \{\{\alpha_i^j\}_{j=1}^{n_i}{\}_{x_i \in D}})$.
        \end{definition}
        We will also require the notion of a coparabolic bundle, which in turn requires the definition of a parabolic sheaf:  Let $\RR$ denote the category whose objects are real numbers, and the morphisms are given by $i^{\alpha, \beta}: \alpha \to \beta$ if $\alpha \geq \beta$. Let $\shm od_{\sho_X}$ denote the category of coherent sheaves of $\sho_X$-modules. An $\RR$-filtered $\sho_X$ module is a functor $E: \RR \to \shm od_{\sho_X}$. The functor $E$ will be denoted $E_\star$, and if the subscript is specified, we denote the application to $E$ to that subscript, i.e. $E_\alpha := E(\alpha)$.

        We define $E[\alpha]$ to be the $\RR$-filtered coherent $\sho_X$-module given by $E[\alpha]_\beta := E_{\alpha + \beta}$ and $i_{E[\alpha]}^{\beta, \gamma} := i_E^{\beta + \alpha, \gamma + \alpha}.$ We have a natural transformation $i_{E}^{[\alpha,\beta]}: E[\alpha]_\star \to E[\beta]_\star$, where $i_E^{[\alpha,\beta]}(\gamma) = i_E^{\alpha + \gamma, \beta + \gamma}.$
        \begin{definition}
            A parabolic sheaf on $C = {X}\setminus D$ with respect to $D$ is an $\RR$-filtered $\sho_X$-module $E_\star$, together with an isomorphism 
            $$j_E: E_\star \otimes \sho_X(-D) \stackrel{\sim}{\to} E[1]_\star$$
            such that $i_E^{[1,0]} \circ j_E = \id_{E_\star} \otimes i_D$, where $i_D: \sho_X(-D) \to \sho_X$ is the inclusion. 
        \end{definition}

        \begin{definition}
            If $E_\star$ is a parabolic vector bundle, viewed as a parabolic sheaf, then the associated coparabolic bundle, denoted $\hat{E}_\star$, is the parabolic sheaf defined by 
            $$\hat{E}_\alpha := \colim_{\beta > \alpha} E_\beta,$$
            where the colimit is taken over all inclusion $i_E^{\alpha, \beta}$. 
        \end{definition}
        
        \begin{definition}
            If $E_\star = (E, \{E_j^i\}, \{\alpha_j^i\})$ is a parabolic bundle, then the parabolic degree of $E_\star$ is given by
            $$\pardeg(E_\star) = \deg E + \sum_{i=1}^n \sum_{j=1}^{n_j} \alpha_j^i \dim(E_j^i/E_j^{i+1}).$$

            We then define the parabolic slope to be $\mu_\star(E_\star) = \pardeg(E_\star)/\rk(E_\star).$
        \end{definition}

        Parabolic bundles are known to admit a notion of (semi)stability, hence they admit a Harder-Narasimhan filtration with essentially the same construction as for an ordinary vector bundle on a curve \cite{sesh}[Part Three, I].

        Parabolic bundles also have a cohomology theory that agrees with the underlying vector bundle, see \cite{Yokdoi:10.1142/S0129167X95000092}. In general, we have $H^i(X, E_\star) = H^i(X, E_0)$. In particular, they satisfy a form of Serre duality, where the tensor product and dual occur in the category of parabolic bundles. Again, see \cite{Yokdoi:10.1142/S0129167X95000092}, or \cite[2.6]{landesman2025geometriclocalsystemsgeneral}.

        \begin{proposition}[Serre Duality]
            Let $X$ be a smooth projective $n$-dimensional variety over an algebraically closed field with dualizing sheaf $\omega_X$. If $E_\star$ is a parabolic bundle on $X$ with respect to a smooth divisor $D$, we have the following canonical isomorphism:
            $$H^i(X, E_\star) \simeq H^{n-i} (X, \widehat{E}_\star^\vee \otimes \omega_X(D)).$$
        \end{proposition}

\section{Background on the Non-Abelian Hodge Correspondence and Variations of Hodge Structure}
We now review the non-abelian Hodge correspondence. We first describe the correspondence between Higgs bundles and flat vector bundles through the intermediate category of tame harmonic bundles. We then describe how the notion of a variation of Hodge structure fits into this correspondnce. We describe the non-abelian Hodge correspondence at the level of moduli spaces, and we end by discussing the Tannakian perspective on variations of Hodge structure. Our presentation follows \cite{brunebarbe2017semipositivityhiggsbundles}[2], \cite{arapuraparabolic}[5], and \cite{simpsonhiggs}[6].

\subsection{Flat Vector Bundles, Higgs Bundles}
Let $X$ be a complex manifold, and let $E$ be a vector bundle on $X$. 
\begin{definition}
	Let $\lambda \in \CC$. A $\lambda$-connection on $E$ is a $\CC$-linear map $D_\lambda: E \to E \otimes_{\sho_X} \Omega_X$ such that $$D_\lambda(fs) = fD_\lambda(s) + \lambda df\wedge s$$ for sections $f$ and $s$ of $\sho_X$ and $E$, respectively.
\end{definition}

When $X \subset \bar{X}$ is the complement of a simple normal crossings divisor $D$, we say that a bundle $E$ on $\bar{X}$ has a logarithmic $\lambda$ connection if instead the $\lambda$-connection lands in $E \otimes_{\sho_X} \Omega_X(\log D)$. If $E_\star$ is a parabolic bundle on $\bar{X}$, we say that $E_\star$ has a meromorphic $\lambda$-connection if the underlying vector bundle $E$ has a $\lambda$-connection landing in $E \otimes_{\sho_X(-D)} \Omega_X (-D)$ and that for every piece $E_\alpha$ of the parabolic filtration on $E_\star$, $D_\lambda$ induces a logarithmic $\lambda$-connection on $E_\alpha$. 

The category of $\lambda$-connections subsumes the categories of flat vector bundles and Higgs bundles. Indeed, when $\lambda = 1$, one recovers the original Leibniz rule for a connection, while when $\lambda=0$, one recovers the defining property of a Higgs field:

\begin{definition}
	A Higgs bundle is a holomorphic vetor bundle $(E, \bar{\partial})$ with a section $\theta \in H^0(\bar{X}, \End(E) \otimes \Omega_X^1)$. 
\end{definition} 

One can go back and forth between a Higgs bundle and a flat vector bundle through the notion of a pluriharmonic metric. Starting on the Higgs side, there is a natural hermitian metric $h$ on $E$, and one can form an adjoint $\theta^\dagger$ to $\theta$ with respect to $h$. Then one defines an connection $\Theta:= \partial + \bar{\partial} + \theta + \theta^\dagger$, which lives on the bundle $E \otimes_{\sho_X} \mathcal{C}^\infty$. On the other hand, if $(E, \nabla)$ is a flat vector bundle, its underlying smooth bundle has a hermitian metric as well, and the connection $\nabla$ decomposes into two operators $\nabla = \nabla^u + \Psi$, where $\nabla^u$ is unitary. Each of these in turn splits as $\nabla^u = \partial + \bar{\partial}$ and $\Psi = \theta + \bar{\theta},$ and one obtains an operator $D'' : = \bar{\partial} + \theta$. One says that a hermitian metric is \textit{pluriharmonic} if either of the operators $\Theta$ or $D''$ are integrable, i.e. $\theta^2$ and $(D'')^2$ are zero. In this case, a pluriharmonic metric on a Higgs bundle gives rise to a flat vector bundle and vice versa. Though they end up being equivalent, when specifying a harmonic bundle, one must either start with a Higgs bundle or a flat vector bundle, and we will start with a Higgs bundle. 
\begin{definition}
	Let $\bar{X}$ be a compact complex manifold and let $X \subset \bar{X}$ be the complement of a simple normal crossings divisor. A harmonic bundle $(E, \theta)$ on $X$ is called \textit{tame} if there exists a logarithmic Higgs bundle on $\bar{X}$ extending $(E, \theta)$. 
\end{definition}

Important in our applications will be those tame harmonic bundles which underly a complex variation of Hodge structure. 
\begin{definition}
	Let $X \subset \bar{X}$ be the complement of a simple normal crossings divisor. A polarizable complex variation of Hodge structure on $X$ is a $\mathcal{C}^\infty$ vector bundle $V$ with a direct sum decomposition $V = \bigoplus_{p} V^p$, and a flat connection $\nabla$ satisfying Griffiths transversality:
	$$\nabla(V^p) \subset A^{0,1}(V^{p+1}) \oplus A^{1,0}(V^p) \oplus A^{0,1} (V^p) \oplus A^{1,0}(V^{p-1}).$$ Moreover, we require the pieces $V^p$ to be orthogonal with respect to the natural hermitian form $h$ on $V$, and that $h$ is positive definite on $V^p$ for $p$ odd and negative definite on $V^p$ for $p$ even.
\end{definition}
To this, one associates a holomorphic flat vector bundle $E$ by taking $E = \ker \nabla \otimes_\CC \sho_X$, and one defines the Hodge filtration $F^\bullet$ on $V$ by setting $F^p V = \bigoplus_{i \geq p} V^{i}.$ An object in the non-abelian Hodge correspondence, such as a tame harmonic bundle, is said to underly a polarizable complex variation of Hodge structure if the underlying $C^\infty$ bundle on $X$ is isomorphic to a polarizable complex variation of Hodge structure $(V, \nabla)$.

\subsection{The Moduli Spaces}
	We next recall the primary objects in the non-abelian Hodge correspondence. If $f: X \to S$ is a smooth proper morphism of schemes over $\CC$, Simpson \cite{PMIHES_1994__80__5_0} has constructed three moduli spaces extending the classical Riemann-Hilbert correspondence. The relative de Rham moduli space $\shm_{dR} (X/S, r)$ is the relative moduli space of rank $r$ flat vector bundles $(E,\nabla)$, while the Dolbeault moduli space $\shm_{Dol}(X/S, r)$ is the relative moduli spaces of rank $r$ Higgs bundles. These moduli spaces are more generally part of a family of moduli spaces of $\lambda$-connections fibered over $\AA^1_\lambda$, and they are the fibers over $1$ and $0$, respectively. The relative Betti moduli space $\shm_B(X/S, \GL_r(\CC))$, on the other hand, is a local system of schemes over $S$, essentially defined by the property that $\shm_B(X/S, \GL_r(\CC))_s = \shm_B(X_s, \GL_r(\CC))$, where the latter variety is the character variety, the GIT qotient $\Hom(\pi_1(X_s, *), \GL_r(\CC))//\GL_r$. The non-abelian Hodge correspondence states that the relative de Rham and Betti moduli spaces are complex analytically isomorphic, while the de Rham and Dolbeault moduli spaces are homeomorphic.
	
	In this paper, we deal with flat vector bundles $(E,\nabla)$ on punctured curves, where the presence of boundary introduces complications. Let $f: \bar{X} \to \bar{S}$ again be a smooth proper morphism, but now where $X \subset \bar{X}$ is the complement of a simple normal crossings divisor $D$. By work of Simpson\cite{simpsonharmonic}, and greatly generalized by Mochizuki \cite{mochizukiharmonic}, there exists a categorical equivalence between polystable parabolic Higgs bundles of degree $0$ with regular singularities along $D$ and polystable parabolic vector bundles of degree zero with a connection with regular singularities along $D$. However, in the non-compact case, there is no longer an analytic isomorphism between the de Rham space and the Betti space. Indeed, given a representation $\rho: \pi_1(X, x) \to \GL_r(\CC)$, the choice of the fundamental domain of the complex exponential gives many choices for an extension of the corresponding representation to $\bar{X}$ as a logarithmic vector bundle and as a parabolic bundle (see below). To remedy this, we impose the condition of \textit{non-resonance}, i.e. that the eigenvalues of the residues of $\nabla$ do not differ by nonzero integers. Such a condition gives a local bijection between the relative space of representations of $\pi_1(X,x)$ and the relative flat vector bundles on $X$ with regular singularities along $D$.
	
	The isomonodromy foliation we consider is on  $\shm_{dR}(\shc_{g,n}/\shm_{g,n}, \shd, r)$, the relative de Rham space of logarithmic vector bundles on an $n$-pointed genus $g$ curve. As described in the next section, a vector bundle on a punctured curve with regular singularities at the punctures has a natural extension as a logarithmic bundle on the completed curve. The connection in turn gives this extension the structure of a parabolic bundle, and we will use this structure (as well as the structure of the associated parabolic Higgs bundle) to conduct our numerical analysis, but the foliation itself lives on the space of logarithmic connections, not on a space of parabolic connections.

\subsection{Tannakian Formalism}
Let $(V, \nabla)$ be a variation of Hodge structure on the complement of a simple normal crossings divisor $X = \bar{X}\setminus D$, where $\bar{X}$ is a smooth compact variety. Let $\VV = \ker \nabla$ be the corresponding local system, and let $T(\VV)$ be the neutral Tannakian category generated by $\VV$. Let $x \in X$ be a basepoint, so that $\VV$ corresponds to a representation $\rho: \pi_1(X,x) \to \GL_r(\CC)$, and let $\omega: T(\VV) \to \text{Vec}_\CC$ be the functor $\VV \mapsto \VV_x$. 
\begin{definition}
	The algebraic monodromy group of $\VV$ is the Zariski closure of the identity component of the image of $\rho$. Equivalently, it is the Tannakian group associated to $T(\VV)$. 
\end{definition}

If $G$ is the algebraic monodromy group of a variation of Hodge structure, then the Lie algebra of $G$ is denoted by $\mathfrak{g}$.

\begin{lemma}\label{subvhsend}
	Let $(E_\star, \nabla)$ be a parabolic bundle on $\bar{X}$ which underlies a variation of Hodge structure. If $G$ is the corresponding algebraic monodromy group, then $\mathfrak{g}$ is a holomorphic subbundle of $\End(E_\star)$, and is in fact a  sub-variation of Hodge structure. 
\end{lemma}
\begin{proof}
	By definition $G$ fiberwise fixes all sub-local systems of $\VV^{\otimes m} \otimes (\VV^\vee)^{\otimes n}$ for all $m,n \geq 0$, i.e. if $\WW \subset \VV^{\otimes m} \otimes (\VV^\vee)^{\otimes n}$ then $G$ is the set of $g \in \GL(\VV_x)$ such that $g(\WW_x) \subset \WW_x$. Differentiating this last equation at the identity says that $\mathfrak{g}$ is a sub-local system of $\End(E_\star)$. If $\End(E_\star)$ has a Hodge filtration $F^\bullet$, then $\mathfrak{g}$ inherits a filtration by setting $F^p \mathfrak{g} = F^p\End(E_\star) \cap \mathfrak{g}$. 
\end{proof}
\section{Background on Isomonodromy}
    In this section we give an overview of the theory of isomonodromic deformations of flat vector bundles with regular singularities and the isomonodromy foliation on the relative de Rham moduli space. See \cite{landesman2025geometriclocalsystemsgeneral}, \cite{chen2011associatedmapnonabeliangaussmanin}, and \cite{biswas2019isomonodromicdeformationslogarithmicconnections} for more details on these topics. 
    \subsection{Atiyah bundles and deformation theory of connections}
    \begin{definition}
        Let $C$ be an algebraic variety with a vector bundle $E$. The sheaf of first-order differential operators, $\Diff^1(E, E)$, is the sheaf of $\CC$-linear maps $\delta \in \End(E)$ such that the map
        $$s \mapsto \delta_f(s) := \delta(fs) - f\delta(s)$$ is $\sho_X$-linear for all local sections $s$ of $E$ and $f$ of $\sho_C$. The Atiyah bundle $\At_X(E) \subset \Diff^1(E,E)$ is the subsheaf of those $\delta$ for which the map $\delta_f$ is locally given by multiplication by a section. 
    \end{definition}

    There is a natural short exact sequence
    $$0 \to \End(E) \to \At_X(E) \to T_X \to 0,$$
    and a connection on $E$ is the same thing as a splitting of this sequence.  Moreover (see \Cite[Definition 3.2.4, Remark 3.2.5]{landesman2025geometriclocalsystemsgeneral}), the Atiyah bundle exists in the more general setting of a filtered parabolic bundle $(E_\star, P^\bullet)$ on an algebraic variety $X$ with a smooth simple normal crossings compactification $X = \bar{X} \setminus D$. In this case, we have a short exact sequence
    $$0 \to \End(E_\star, P^\bullet) \to \At_{(X, D)} (E_\star, P^\bullet) \to T_X (-D) \to 0.$$

	Given a filtered vector bundle $(E, P^\bullet)$ on $X$ and a connection $\nabla$ with regular singularities along $D$, there is a canonical parabolic structure on $E$ induced by the connection, called the \textit{Deligne canonical extension}. It is characterized by the property that it is the unique holomorphic bundle on $\bar{X}$ that when restricted to $X$ is isomorphic to $(E,\nabla)$, and has all of the eigenvalues of the residues of its connection contained in the interval $[0,1).$ 
    
    By comparing this short exact sequence to the analogous sequence for an unfiltered bundle, we get a short exact sequence
   \begin{equation}\label{ATSeq}
        0 \to \At_{(X,D)} (E_\star, P^\bullet) \to \At_{(X, D)} (E_\star) \to \End(E_\star)_\star/\End(E_\star, P^\bullet)_\star \to 0.
   \end{equation}

    \subsection{Isomonodromic deformations and the isomonodromy foliation}
    
    To introduce isomonodromic deformations, we relativize the above construction to the case of a relative curve $X/S$. In this case, a connection $\nabla$ on a vector bundle $E$ over $X$ lets us form the \textit{de Rham complex} of $\End(E)$:
    $$\DR(\End(E))^\bullet := \End(E) \stackrel{\nabla}{\to} \End(E) \otimes \Omega_X^1 (\log D) \stackrel{\nabla}{\to} \End(E) \otimes \Omega^2_X (\log D) \to \dots.$$
    Likewise, one forms the de Rham complex of the Atiyah Bundle:
    $$\DR(\At_{E})^\bullet := \At_X (E) \to \End(E) \to \End(E) \otimes \Omega_X^1 (\log D)\to \End(E) \otimes \Omega_X^2 (\log D) \to \dots .$$
    There is a short exact sequence of complexes:
    $$0 \to \DR(\End(E))^\bullet \to \At_X(E)^\bullet \to T_X (-D) \to 0.$$

    Given $(E,\nabla)$ on $X$, the natural section of $\DR(\At_X(E))^\bullet \to T_X (-D)$ given by $\nabla$ induces a corresponding section $\sigma$ of $\HH^1(X, \At_X(E)^\bullet) \to H^1(X, T_X(-D))$. On the other hand, for $s \in S$, we have the Kodaira-Spencer map $\kappa: T_sS \to H^1(X, T_X(-D))$. We may take the fiber product of these two maps, and the connection induces a section $\sigma: T_sS \to T_{X, (E, \nabla}) \shm_{dR}(X/S, r)$: 
\[\begin{tikzcd}
	{T_{(X, (E, \nabla)} \shm_{dR}(X/S, r)} & {\HH^1(\DR(\At_X(E)^\bullet)} \\
	{T_sS} & {H^1(X, T_X(-D))}
	\arrow[from=1-1, to=1-2]
	\arrow[from=1-1, to=2-1]
	\arrow[from=1-2, to=2-2]
	\arrow["\kappa"', from=2-1, to=2-2]
	\arrow["\sigma"', shift right=2, curve={height=6pt}, dashed, from=2-1, to=1-1]
\end{tikzcd}\]
\begin{definition}
    Let $(\she, \nabla)$ be a flat vector bundle over $\shm_{g,n}$. The \textit{isomonodromy foliation} on $\shm_{dR} (\shc/\shm_{g,n}, \shd, r)$ is defined by the section $H^1(C, T_C) \to T_{\shc, (\she, \nabla)} \shm_{dR} (\shc/\shm_{g,n}, r)$ described above. 
\end{definition}

This is the local description of the isomonodromy foliation. Turning to the global description, the leaves of the foliation are described by the following proposition. To set notation, let $\pi: \shc \to \sht_{g,n}$ be the universal family over Teichmuller space, and let $s_1, \dots, s_n: \sht_{g,n} \to \shc$ be disjoint sections, with image $\shd$. Let $0 \in \sht_{g,n}$ and set $C = \pi\inv(0) \setminus D$, where $D = \shd \cap \pi\inv(0)$. 

\begin{lemma}
	Let $(E, \nabla)$ be a flat vector bundle on $C$ with regular singularities along the divisor $D$ of the punctures. If the eigenvalues of the residues of $\nabla$ along the points of $D$ do not differ by consecutive integers, then the inclusion $C \injto \shc \setminus \shd$ gives a unique extension of $(E, \nabla)$ to a logarithmic flat vector bundle $(\she, \tilde{\nabla})$ over $\shl$ with regular singularities along $\shd$, such that the residues of the eigenvalues of $\tilde{\nabla}$ lie in $[0,1)$. 
\end{lemma}
	\begin{proof}
		Starting with a logarithmic flat bundle on $C \cup D$ with regular singularities along $D$, \cite{landesman2025geometriclocalsystemsgeneral}[Lemma 3.4.2] gives the construction of $(\she, \tilde{\nabla})$ as a canonical extension of $(E, \nabla)$. The imposition that the eigenvalues of the residues of $\nabla$ do not differ by consecutive integers ensures that there is a unique extension of the corresponding representation $\rho: \pi_1(C) \to \GL_r(\CC)$ of the punctured curve as a logarithmic bundle \cite{bakker2024linearshafarevichconjecturequasiprojective}[Proposition 7.4].
	\end{proof} 
	 If the restriction of $(\she, \nabla)$ to $(C,D)$ is $(E, \nabla)$, we call $\shl$ the \textit{universal isomonodromic deformation} of $E$. The connection on $\she|_{C_t}$ for a very general fiber $C_t$ induces a natural parabolic structure on $\she|_{C_t}$, and this in turn induces a parabolic structure on the isomonodromic deformation $\she$. We denote the associated parabolic bundle by $\she_\star$, see \cite[Section 4.3]{biswas2019isomonodromicdeformationslogarithmicconnections}.

    \subsection{The Simpson Filtration and its Deformations}
    The isomorphism between the de Rham and the Dolbeault spaces given by the nonabelian Hodge correspondence is only real-analytic. Therefore if one applies this isomorphism to a component $Z$ of the locus of $\CC$-VHSs within $\shm_{Dol}(\shc/\shm, \shd, \GL_r)$, one would only obtain a real analytic subset of the de Rham space. Important for our main theorem is the fact that, when such a component is intersected with a leaf of the isomonodromy foliation, the resulting locus is actually complex analytic. We will actually prove a stronger result: such a statement is true not just for the locus where the Hodge filtration extends to the isomonodromic deformation, but in fact it holds for the locus where the Simpson filtration extends. 
    
    To begin, we must recall the construction of the Simpson filtration. A Simpson filtration \cite{simpson2008iterateddestabilizingmodificationsvector} on a flat vector bundle $(E, \nabla)$ is a decreasing Griffiths-transverse filtration $F^\bullet$ whose associated graded $(\bigoplus \gr^p_F E, \theta^p)$ is a semistable Higgs bundle (we then say that $E$ is $gr$-stable with respect to $F^\bullet$). It is generally not unique, though its associated graded pieces are. The Simpson filtration also has a notion of level defined in the same way as the level of the Hodge filtration. Importantly, if $(E, \nabla)$ underlies a $\CC$-VHS, then the Hodge filtration is a Simpson filtration. Simpson's filtration was first constructed in the context of flat vector bundles on proper algebraic varieties, but it has been shown to exist in the more general setting of parabolic bundles with respect to a smooth irreducible divisor, see \cite[Appendix A]{collier2024conformallimitsparabolicslnchiggs}.

   \begin{remark}One goal of the Simpson filtration was to provide a stratification of the de Rham moduli space. Indeed, letting $\Hdg \subset \shm_{Dol}(\shc/\shm_{g,n}, \shd \GL_r)$ be the image of the fixed point set of the natural $\GG_m$ action on the Dolbeault space, we may write 
   $$\Hdg = \bigcup_\alpha P_\alpha.$$

   We may then define a stratification of $\shm_{dR}(\shc/\shm_{g,n}, \GL_r)$ by defining $G_\alpha$ to be all the points $y$ whose limit as $t \to 0$ (in the sense of $\lambda$-connections) lies in $P_\alpha$. 
   \end{remark}
   
    \begin{lemma}\label{analyticsub}
        Let $\shl$ be a leaf of the isomonodromy foliation on $\shm_{dR}(\shc/\shm_{g,n}, \mathcal{D}, \GL_r)$, corresponding to the universal isomonodromic deformation of a logarithmic vector bundle $(E,\nabla)$ on an $n$-pointed curve. Let $\shz \subset \shm_{Dol}(\shc/\shm, \mathcal{D}, \GL_r)$ be the locus of gr-stable Higgs bundles. Let $\Phi: \shm_{Dol}(\shc/\shm, \shd, \GL_r) \to \shm_{dR}(\shc/\shm, \shd, \GL_r)$ be the real analytic isomorphism given by the nonabelian Hodge correspondence. Let $F^\bullet$ be a Simpson filtration on $(E_\star, \nabla)$, and let $\shu \subset \Phi(\shz) \cap \shl$ be the locus for which $F^\bullet$ extends to the universal isomonodromic deformation of $(E, \nabla)$. Then $\shu$ has the structure of a locally closed analytic subset of $\shl$. 
    \end{lemma}
        \begin{proof}
            Let $\sht$ be the universal cover of $\shl$. Let $\she$ be the universal isomonodromic deformation of a logarithmic bundle $(E, \nabla)$ over $\sht$, and let $\she_\star$ denote the parabolic bundle with parabolic structure induced by the natural parabolic structure on $E$ given by the residues of $\nabla$. By the main result of \cite{Gurjar}, the locus within $\sht$ where the Harder-Narasimhan filtration of the corresponding Higgs bundle extends to $\she$ is a locally closed analytic subset of $\sht$. Hence, we may execute the construction of the Simpson filtration on $\she_\star$, as the Simpson filtration is obtained by iterative quotients of the Harder-Narasimhan filtration of $\gr_{F^\bullet} E_\star$, where $F^\bullet$ is a given Griffiths-transverse filtration on $E_\star$ (we begin with the trivial filtration). 

            Let us recall this construction breifly. If $\gr_{F^\bullet} \she_\star$ is not a semistable Higgs bundle, there is a maximal subbundle of the form $\mathcal{H} = \bigoplus \mathcal{H}^p$, with each $\mathcal{H}^p \subset F^p\she/ F^{p+1}\she_\star$. Then a new filtration $G^\bullet$ may be defined on $\she_\star$ by setting
            $$G^p = \ker\big( \she_\star \to \frac{\she/F^p \she_\star}{\mathcal{H}^{p-1}} \big).$$

            This construction terminates after finitely many steps (\cite{simpson2008iterateddestabilizingmodificationsvector}[Lemma 3.3], see also \cite[Appendix A]{collier2024conformallimitsparabolicslnchiggs} for the parabolic variant), and the resulting filtration is the Simpson filtration. For each intermediate filtration $G^\bullet$, the locus within $\sht$ for which the Harder-Narasimhan filtration of $\gr_{G^\bullet} \she_\star$ extends to the isomonodromic deformation is a locally closed analytic subset, hence the same is true for the finite intersection of all such loci. 
        \end{proof}

	Let us also justify our claim that a Simpson filtration records some data regarding semistability of a holomorphic bundle. The following result is essentially the same as \cite{landesman2025geometriclocalsystemsgeneral}[Lemma 4.1.5]. 
	
	\begin{lemma}\label{SimpsonDestabilizes}
		Let $(C,D)$ be an $n$-pointed hyperbolic curve, and let $(E_\star, \nabla)$ be a parabolic bundle on $C$ with respect to $D$, and let $F^\bullet$ be a Simpson filtration on $E_\star$. Let $i$ be the largest $i$ such that the map $\theta_i: F^i E_\star \to \gr_F^{i-1} E_\star \otimes \omega_C(D)$ is nonzero. If the connection $\nabla$ is irreducible, then $F^i E_\star$ has positive parabolic degree. In particular, $E_\star$ is not semistable if $i >0$.
	\end{lemma}
		\begin{proof}
			By the definition of the Simpson filtration, the graded Higgs bundle $\gr_F E_\star$ is semistable of parabolic degree zero. The bundle $F^i E_\star$ is a parabolic quotient of $\gr_F E_\star$, so it has non-negative degree. Since the connection was supposed to be irreducible, the graded Higgs bundle is in fact polystable of degree zero, so $F^i E_\star$ has degree zero if and only if it is a summand of $\gr_F E_\star$, but this is not possible if $\theta_i \neq 0$. Since $F^i E_\star$ has positive parabolic degree while $E_\star$ has parabolic degree zero, it follows that $F^i E_\star$ destabilizes $E_\star$. 
		\end{proof}
	
	\begin{remark}
		One can also argue \Cref{SimpsonDestabilizes} directly from the construction of the Simpson filtration. Indeed, $F^\bullet$ is obtained by iteravely destabilizing a Griffiths-transverse filtration on the flat bundle $E_\star$, then defining a new filtration that incorporates the destabilizing subsheaves while preserving transversality. The trivial filtration $0 = F^1 \subset F^0 = E_\star$ is a Griffiths-transverse filtration, and if one runs Simpson's algorithm starting with this filtration, the graded Higgs bundle is equal to $E_\star$, so a destabilizing Higgs bundle is a destabilizing flat subbundle. Thus a nontrivial Simpson filtration can only be produced from an unstable parabolic bundle. 
	\end{remark}
    \subsection{Cohomological Interpretation of the Deformation theory of a Flat Bundle}
    Implicit in the relative construction above is the fact that the deformation theory of a flat vector bundle $(E,\nabla)$ on a fixed curve is controlled by $\HH^1(C, \DR(\End(E))$. This in fact extends to the parabolic setting, where we will work, see \cite[Prop. 5.2.1]{Bottacin}. Indeed, one may form the (parabolic) de Rham complex:
    
    $$\DR(\End(E_\star))^\bullet := \End(E_\star)_\star \to \End(E_\star) \otimes \Omega_C^1 (D) \to \End(E_\star) \otimes \Omega_C^2 (D) \to \dots.$$
    
     Then if $(E_\star, \nabla)$ is irreducible, the deformations are unobstructed, and the deformation space is exactly the $\HH^1$ of this complex. Moreover, if $(E_\star,\nabla)$ is equipped with a Griffiths-transverse filtration $F^\bullet$, there is an induced filtration of $\DR(\End(E_\star))$ and a spectral sequence

    $$\HH^i (\gr^p \DR(End (E_\star))) = E_1^{p, i-p} \implies \gr^p \HH^i(\DR(\End(E_\star))).$$

    For our purposes we must also consider the deformation theory of the associated Higgs bundle $(\gr_F E_\star, \theta).$ We can form the de Rham complex $\DR(\gr_F E_\star)$, just as well, but this time the differentials are the Higgs field. It follows that the deformation theory is controlled by $\HH^1(\DR(\End(\gr_F E_\star)_\star)).$  In particular, the spectral sequence above computes the deformation space of the associated graded Higgs bundle. 
    
    \begin{lemma}{(Cf. \cite[Lemma 7.1]{simpson2008iterateddestabilizingmodificationsvector})}
        Let $(E_\star, \nabla)$ be a parabolic bundle on a curve $C$ with respect to a reduced divisor $D = x_1 \dots + x_n.$ If $F^\bullet$ is a Griffiths-transverse filtration on $(E_\star,\nabla)$, then there is a natural spectral sequence $\HH^i (\gr_F^p \DR(\End(E_\star))) = E^{p,i-p}_1 \implies \gr^p_F \HH^i(\DR(\End(E_\star))).$ If $(\gr_F E_\star, \theta)$ is stable, then this spectral sequence degenerates at $E_1$. 
    \end{lemma}
        \begin{proof}
            Simpson only states this for stable Higgs bundles, but with the notion of the Simpson filtration in place for parabolic bundles, the same proof works in this new setting. Since $(\gr_F E_\star, \theta)$ is stable, we have $\HH^0(\DR(\End(E_\star))) = \HH^2(\DR(\End(E_\star))) = \CC$. The Euler characteristic of the complex $\DR(\End(E_\star))$ is the same before and after taking the associated graded complex. Since the $E_1$ page of the spectral sequence computes $\HH^1$ of the associated graded complex, all the differentials are forced to be zero. 
        \end{proof}
        Retaining the notation, we have the following:
    \begin{corollary}\label{SSdegeneratesprojection}
        If $(\gr_F E_\star, \theta)$ is stable, then there is a natural surjection 
        $$\HH^1 (\DR(\End(E_\star))) \surjto \bigoplus_{i=1}^\ell H^1(\Hom(\gr^i E_\star, \gr^{i-1} E_\star)).$$
    \end{corollary}
        \begin{proof}
            The degeneration of the spectral sequence implies that $\HH^1(\DR(\End(E)))$ is graded by the terms on the first page, so we may project to the $\gr^{-1}$ piece. 
        \end{proof}

    Important to our work will be the following deformation-theoretic interpretation of the long exact sequence in cohomology of the Atiyah bundle exact sequence. 
    \begin{lemma}{Cf. \cite[Lemma 3.5.8]{landesman2025geometriclocalsystemsgeneral}}\label{codimrank}
        Let $(C,D)$ be a smooth proper curve with an $n$-pointed reduced effective divisor $D$, and let $(E_\star, \nabla)$ be a flat parabolic bundle with regular singularities along $D$. Let $F^\bullet$ be a filtration of $E_\star$. If $s \in H^1(C, T_C(-D))$ is a first-order deformation of $C$ whose image $q^\nabla(s)$ in $\At_{(C,D)} (E_\star)$ under $\nabla$ is a deformation of $(C,D,E_\star)$ for which $F^\bullet$ extends to the isomonodromic deformation $\she$, then 
        $q^\nabla(s) \in \ker(\HH^1(C, \At_{(C,D)}(E_\star) \to \HH^1(C, \shend(E_\star)/\shend(E_\star, F^\bullet)).$
    \end{lemma}
        \begin{proof}
             First, note that one may form the de Rham complex for the Atiyah bundle of a filtered vector bundle:
            $$\DR(\At_{E, F^\bullet})^\bullet := \At_{(C,D)} (E, F^\bullet) \to \End(E, F^\bullet) \to \End(E, F^\bullet) \otimes \Omega_C^1.$$
            The long exact sequence in hypercohomology associated to the short exact sequence of complexes $0\to \DR(\At_{E, F^\bullet})^\bullet \to \DR(\At_)^\bullet \to \DR(\shend(E_\star)/\shend(E_\star, F^\bullet)$ induces a map 
            $$\HH^1(C, \DR(\At_{E_\star})^\bullet) \to \HH^1(C, DR(\shend(E_\star)/\shend(E_\star, F^\bullet)).$$
            Since $q^\nabla(s) \in \HH^1(C, \At_{E_\star}^\bullet)$ is in the image of 
            $$\HH^1(C, \DR(\At_{E_\star, F^\bullet})^\bullet) \to \HH^1(C, \DR(\At_{E_\star})^\bullet),$$
            its image in $\HH^1(C, \shend(E_\star)/\shend(E_\star, F^\bullet))$ vanishes.
        \end{proof}

    In particular, if the filtration $F^\bullet$ extends to germ of a codimension $\delta$ subvariety, then it follows that the cokernel of the above map on $H^1$ has rank $\delta$.

\section{Proofs}
    In this section we will prove the main theorems of the paper. We prove \Cref{rankbound} in two parts. First, we prove \Cref{parabolicrkbound}, the geometric content of \Cref{rankbound}, which is a parabolic variant of the main theorem with a slightly different rank bound. We then optimize this bound in \Cref{convexarg}. Finally, \Cref{nhldimension} will follow quickly from \Cref{rankbound}.  
    
\subsection{Outline of the Proof When $\delta = n = 0$}
	In this subsection we will give an abbreviated argument for \Cref{rankbound} in the simplest case when $\delta = n = 0$. This is only enough to recover \Cref{semistability} in the case of a curve with no marked points, but it should be enough to convey the idea of the main calculation while avoiding the complications introduced by working with parabolic bundles and having nonzero maps on cohomology. It may be worthwhile to compare this argument to \cite{landesman2025geometriclocalsystemsgeneral}[Theorem 1.3.4].
	
	Let $(E, \nabla)$ be a flat vector bundle on a very general curve with a Simpson filtration $F^\bullet$, and suppose that $F^\bullet$ extends to the very general curve in a leaf of the isomonodromy foliation. In this case, steps 1-4 of the proof (see the proof \Cref{parabolicrkbound} below) follow verbatim, so by chasing the Atiyah bundle exact sequence, we get a nonzero map $\psi: T_C \to \gr_F^{-1} \End(E)$ inducing the zero map on $H^1$. 
	
	Dually, we get a map $\psi^\vee: \gr^1_F \End(E) \otimes \omega_C \to \omega_C^2$ which vanishes on $H^0$. This map kills the image of the Higgs field $\ad\theta$ on $\End(E)$, so letting $(Q,0)$ be the cokernel of $\ad \theta: \gr_F^2 \End(E) \to \gr_F^1 \End(E) \otimes \omega_C$, we get a nonzero map $\varphi: Q \to \omega_C^2$ which vanishes on $H^0$. 
	
	It follows that $Q$ is not generically globally generated, so $\varphi$ kills any subbundle of $Q$ with Harder-Narasimhan slopes above $2g-1$. Letting $N_\bullet$ be the Harder-Narasimhan filtration of $Q$, and letting $N_s$ be the largest piece of $N_\bullet$ whose graded piece has slope above $2g-1$, we have a map $Q/N_s \to \omega_C^2$ which vanishes on cohomology. We may get a further map $\bar{\varphi}: N_{s+1}/N_s \to \omega_C^2$ which vanishes on $H^0$ by composing with the inclusion $N_{s+1}/N_s \injto Q/N_s$. 
	
	Since $(Q,0)$ was a quotient Higgs bundle of the stable Higgs bundle $\End(E) \otimes \omega_C$, we have $\mu(Q) > 2g-2$. This bound is also true for $N_{s+1}/N_s$, since the last graded piece of the Harder-Narasimhan filtration of $Q$ is a quotient of $Q$, and the successive graded pieces o f the Harder-Narasimhan filtration have decreasing slopes. Letting $U \subset N_{s+1}/N_s$ be the subbundle of $N_{s+1}/N_s$ generated by global sections, the non-ggg lemma \cite{landesman2025geometriclocalsystemsgeneral}[Proposition 6.3.6, or see \Cref{nonggg} below] says that $\rk N_{s+1}/N_S > g$, and the same is true for $\rk \gr_F^1 \End(E)$. To get a bound on the rank of $E$, and to introduce the level of the filtration into the bound, we use the AM-GM and Cauchy-Schwarz inequalities (see \Cref{convexarg} below).

\subsection{Variations on the non-ggg Lemma}
	Before diving into the proof proper, we collect some results on high slope bundles that we will use to make reductions.
	
	The output of the deformation theory in our argument will be a parabolic bundle of high slope with a nonzero map to $\omega_C^2(D)$ whose image on global sections is nonzero. We first aim to pass this map to a graded piece of the Harder-Narasimhan filtration of $Q_\star$.
	
\begin{lemma}\label{HNQuotient}
	Let $E_\star$ be a parabolic bundle with respect to a divisor $D$ on a hyperbolic curve $(C,D)$, and suppose there exists a nonzero map $\varphi: E_\star \to L$, where $L$ is a line bundle whose image on global sections has dimension $\delta$. Suppose any quotient parabolic bundle of $E_\star$ satisfies $\mu_\star(Q_\star) > 2g-2 + n$, and let $N_\bullet$ be the parabolic Harder-Narasimhan filtration of $E_\star$. Then there exists some positive integer $s$ such that the graded piece $\gr_{s+1}^N E_\star$ has a nonzero map to $L$ whose image on global sections is at most $\delta$-dimensional. 
\end{lemma}
	\begin{proof}
		Let $N_s$ be the largest piece of the Harder-Narasimhan filtration of $E_\star$ which vanishes under $\varphi$. Then by the universal property of quotients, there is a nonzero map $E_\star/N_s E_\star \to L$. Since $\mu_\star(Q_\star) > 2g-2 + n$ for any quotient of $E_\star$, this bound is true for all of the graded pieces of $N_\bullet$. Therefore this lower bound is true for $N_s E_\star$ as well. It follows that the short exact sequence $0 \to N_s E_\star \to E_\star \to E_\star/N_s E_\star \to 0$ is exact on $H^0$, so the map $E_\star/ N_s E_\star \to L$ has at most $\delta$-dimensional image on $H^0$. We then compose this with the inclusion $\gr_{s+1}^N E_\star \injto E_\star/N_s E_\star$ to get the desired map. 
	\end{proof}
	
	With a semistable object in hand, we next twist down by a carefully chosen divisor on $C$ to kill the image of the map on global sections. 
\begin{lemma}\label{killsections}
	Let $E_\star$ be a parabolic bundle with respect to a divisor $D$ on a hyperbolic curve $(C,D)$, and suppose there exists a map $\varphi:E_\star \to L$, where $L$ is a line bundle. Suppose moreover that the image of the map $H^0(E_\star) \to H^0(L)$ has dimension $\delta$. Then there exists a parabolic subsheaf $E_\star' \subset E_\star$ such that the composition $E_\star' \injto E_\star \to L$ vanishes on $H^0.$ 
\end{lemma}
	\begin{proof}
		Given a basis $\{v_1, \dots, v_\delta\}$ of the image of $\varphi$ on $H^0$, pick sections $s_i \in H^0(E)$ that map to $v_i$. We may evaluate $E_\star$ at $\delta$ general points (avoiding the points of $D$) to delete this basis: if $p_i \in C$, choose a complementary hyperplace $H_i \subset E_\star|_{p_i}$, then define $E_\star' \subset E_\star$ to be the subbundle of sections whose values at $p_i$ land in $H_i$. The induced map $H^0(E_\star') \to H^0(L)$ consequentially has rank $0$. 
	\end{proof}
	 
	Notably, the bundle $E_\star'$ is not generically globally generated, as the kernel of $E_\star' \to L$ has the same global sections as $E_\star'$.

	\begin{lemma}\label{goodsubquotient}
		Let $E_\star' \subset E_\star$ be a subbundle of a semistable parabolic bundle and further suppose:
		\begin{enumerate}
			\item The quotient $E_\star/E_\star'$ is torsion, supported at most of length one at $\delta$ points,
			\item Any parabolic quotient $Q_\star$ of $E_\star$ satisfies $\mu_\star(Q_\star) > \beta + n$.
			\item $\mu_\star(E_\star) \leq \alpha + n$.
		\end{enumerate}
		Then $\mu_\star (E_\star')$ is contained in the interval $(\beta - \frac{\delta}{\rk E_\star} + n, \alpha + n].$
	\end{lemma}
		\begin{proof}
			The upper bound on $\mu_\star(E_\star)$ follows immediately from the semistability of $E_\star$. For the lower bound, the exact sequence $0 \to E_\star' \to E_\star \to E_\star/E_\star' \to 0$ shows that $\pardeg(E_\star') + \text{len}(E_\star/E_\star') = \pardeg(E_\star).$ Hence $\mu_\star(E_\star') \geq \mu(E-\star) - \frac{\delta}{\rk E_\star}.$
		\end{proof}

	Given the setup above, we are able to use a variation of the non-ggg lemma that is the main technical input to the rank bound in $\Cref{rankbound}$. 
	
	\begin{remark}
		Our version of the non-ggg lemma differs ever so slightly from the non-ggg lemma that appears in \cite{landesman2025geometriclocalsystemsgeneral}[Proposition 6.3.6], but it is functionally the same and one could complete the proof of \Cref{parabolicrkbound} using either proposition. We only state the lemma differently so that it more closely resembles the setup that arises in the proof of \Cref{parabolicrkbound}.
	\end{remark}
	\begin{proposition}\label{nonggg}
		Let $E_\star$ be a nonzero semistable parabolic bundle of rank $r$ with respect to a divisor $D$ on a hyperbolic curve $(C,D)$. Suppose that $E_\star' \subset E_\star$ is a rank $r$ subbundle which is
		 not generically globally generated,  If $E_\star$ is already not generically globally generated, assume $E_\star = E_\star'$ is semistable. Let $U \subset E_\star'$ be the subsheaf of $E_\star$ generated by global sections.  Suppose that $U$ has corank $c$, and that $\mu_\star(E_\star') \geq \mu_\star(E_\star) - \frac{\delta}{r}$. Then:
		\begin{enumerate}
			\item If $\mu(E_\star) > 2g-2$, then $r > cg -\delta$.
			\item If $\mu(E_\star) = 2g-2$, then $r \geq cg-\delta$.
		\end{enumerate}
	\end{proposition}
		\begin{proof}
			 Since $U$ is a subbundle of $E_\star$, all of its Harder-Narasimhan slopes are bounded above by $2g-1$. It therefore satisfies the hypotheses of Clifford's theorem, and we conclude
			$$h^0(U) \leq \frac{\deg U}{2} + \rk U \leq \frac{\mu(E_\star)(\rk E_\star - c)}{2} + (\rk E_\star - c).$$
			
			In the final step, we have used that $\mu(U) \leq \mu(E_\star)$ by semistability. Since $E_\star'$ is not generically globally generated, we have $h^0(E_\star') = h^0(U)$. Riemann-Roch then says
			$$h^0(E_\star') \geq r (\mu(E_\star') + 1 - g) \geq r(\mu_\star(E_\star) +1 - g) - \delta.$$

			Using that $h^0(E_\star') = h^0(U)$ we may combine the previous two inequalities. Clearing the denominator of two, we have
			$$ r(2\mu(E_\star) + 2 - 2g )- 2\delta \leq (r-c)(\mu(E_\star + 2)).$$
			Simplifying, we have
			$$c (\mu(E_\star) + 2) - 2\delta \leq r(2g - \mu(E_\star))$$
			
			Since $\mu(E_\star) \geq 2g-2$, we have that $2g - \mu(E_\star) \leq 2$, and $\mu(E_\star) + 2 \geq 2g$. Whence
			$$2r \geq 2cg - 2\delta.$$
			Dividing by two obtains the result. We note that if the inequality on the slope $E_\star$ is strict, then so is the final result. 
		\end{proof}
\subsection{Proof of \Cref{rankbound}}
    Here, we will prove:
    \begin{theorem}\label{parabolicrkbound}
        Let $(C,D)$ be a pointed hyperbolic curve equipped with a regular flat vector bundle $(E,\nabla)$ with irreducible monodromy. Let $E_\star$ be the associated parabolic vector bundle given by the residues of $\nabla$. Suppose $(E_\star, \nabla)$ admits an isomonodromic deformation $(E_\star', \nabla')$ for which the Simpson filtration extends to a the germ of a codimension $\delta$ subvariety. If a Simpson filtration $F^\bullet$ of $E_\star'$ has level $\ell$, and the associated graded $(\gr_F E, \theta)$ is stable, then 
        $\rk \gr_F^1 \End(E_\star) \geq g- \delta$. 
    \end{theorem}
    
    \begin{proof}
        We prove the theorem in a number of steps. \\
        1) By \Cref{analyticsub}, the locus within $\shl$ for which the Simpson filtration of $(E_\star, \nabla)$ extends to the isomonodromic deformation has the structure of a locally closed analytic subset. Hence, by replacing $(C,D)$ with an analytically very general nearby curve, and $(E_\star, \nabla)$ with its isomonodromic deformation, we may assume that the Simpson filtration of $E_\star$ extends to $\she_\star$ to the germ of a codimension $\delta$ locus within a first order neighborhood of $[(C,D)]$. 

        2) Since $(E_\star, \nabla)$ extends to the germ of a codimension $\delta$ locus, the connection induces a map of complexes
        $$T_C(-D) \stackrel{q^\nabla}{\to} \DR(\At_{E_\star})^\bullet \to \DR(\shend(E_\star))^\bullet/\DR(\shend(E_\star, F^\bullet))^\bullet.$$ The image of $H^1(C, T_C(-D)) \to \HH^1(C, \DR(\At_{E_\star})^\bullet)$ maps to zero in \newline $\HH^1(\DR(\shend(E_\star))^\bullet/\DR(\shend(E_\star, F^\bullet))^\bullet)$ precisely when $F^\bullet$ extends to a first-order neighborhood of $(C,D)$ by \Cref{codimrank}.
        
        3) Taking $F^\bullet$ to be the Simpson filtration, Griffiths transversality implies that the map 
        $T_C (-D) \to \DR(\shend(E_\star))^\bullet/\DR(\shend(E_\star, F^\bullet))^\bullet$ factors through $\DR(End_{F^\bullet}^{-1}(E_\star))^\bullet,$ where $\End_{F^\bullet}^{-1}(E_\star) = \bigoplus_{p}\Hom(\gr_{F^\bullet}^p E_\star, \gr^{p-1}_{F^\bullet} E_\star).$
        we recognize $\HH^1(\DR(\End_{F^\bullet}^{-1}(E_\star)))$  as the deformation space of $(\gr_F E_\star, \theta)$. Since $(\gr_F E_\star, \nabla)$ is stable, \Cref{SSdegeneratesprojection} gives a surjection
        $$\HH^1(\DR(\End_{F^\bullet}^{-1}(E_\star))) \to \gr_F^{-1} \HH^1(\End(E_\star)) = \bigoplus_{i=1}^\ell \HH^1(\gr_F^i \End(E_\star) \to \gr_F^{i-1} \End(E_\star)).$$ Hence we get a nonzero map $\psi: T_C(-D) \to \gr_F^{-1}(\End(E_\star))$, which induces a map on $\HH^1$ with rank at most $\delta$. 

        4) The Higgs field on $\End(E_\star)$ is $\ad \theta: \gr_F^2\End(E_\star) \otimes T_C(-D) \to \gr_F^1\End(E_\star).$ By adjunction we have a map $\gr_F^2 \End(E_\star) \to \gr_F^1 \End(E_\star) \otimes \Omega_C(D)$. Let $Q_\star$ be the cokernel of this map. Since $(\End(E_\star), \ad\theta)$ is a stable parabolic Higgs bundle of parabolic slope $0$, the quotient Higgs bundle $(Q_\star, 0)$ is a quotient parabolic Higgs bundle with parabolic slope $\mu_\star(Q) > 2g-2 + n$. 
        
        5) Taking the Serre dual of $\psi$, we get a nonzero map $\psi^\vee :\gr_F^1 \widehat{\End(E_\star)} \otimes \omega_C (D) \to \widehat{\gr_F^0\End(E_\star)} \otimes \omega_C^2 (D)$. Importantly, since $[\theta,\theta] = 0$, $\psi^\vee$ factors through $Q_\star$. The Grothendieck spectral sequence for $\HH^*(\gr_F^2\End (E_\star) \to \gr_F^1 \End E_\star \otimes \Omega(D))$ degenerates on the second page, giving a surjection $\HH^1(\gr_F^2 \End(E_\star) \to \gr_F^1 \End(E_\star) \otimes \Omega (D)) \surjto H^0(Q)$. Thus the resulting map $H^0(Q_\star) \to H^0((\Omega^2)$ has rank at most $\delta$. Henceforth, we abuse notation and write that the rank of this map is still $\delta$, as a drop in rank would only improve our result.

        6) By \Cref{HNQuotient}, there exists a Harder-Narasimhan graded piece $N_\star$ of $Q_\star$  such that the induced map $N_\star \to \omega_C^2(D)$ is nonzero and whose image on $H^0$ is at most $\delta$. Moreover, by \Cref{killsections} there is a subsheaf $M$ of $N_\star$ of the same rank such that composition $M \injto N_\star \to \omega_C^2(D)$ vanishes on $H^0$ . By \Cref{goodsubquotient}, the parabolic slope of $M$ is at least $2g - 2 +n - \frac{\delta}{\rk M}$. By \cite{landesman2025geometriclocalsystemsgeneral}[Lemma 6.3.4], we have $\mu(M) > 2g- 2 - \frac{\delta}{\rk M}$.

		7) By \Cref{nonggg}[2] (noting that the subsheaf $U$ generated by global sections has corank at least one), we conclude that $\rk N_\star > g - \delta.$ Since this bound also holds for its saturation, which is a quotient of $\gr_F^1(\End(E_\star)) \otimes \Omega_C,$ the conclusion applies to $\gr_F^1(\End(E_\star))$ as well. 
    \end{proof}
  	
  \subsubsection{The bound on $\rk E_\star$}
    The result of \Cref{parabolicrkbound} was a bound on the rank of $\gr_F^1\End(E_\star)$. In this subsection, we will deduce from this a bound on the rank of $E_\star$, proving \Cref{rankbound}. 

\begin{corollary}\label{convexarg}
	Let $r = r_0 + \dots + r_\ell$ be a partition of a positive integer $r$ into $(\ell + 1)$ positive integers. Suppose $C$ is a quantity such that $\sum_{i=0}^{\ell - 1} r_{i+1} r_i \geq C$. Then $ r \geq \sqrt{C - \ell} + \ell$. 
	
	In particular, with $(E_\star, \nabla)$ as in the conclusion to \Cref{parabolicrkbound}, we may conclude that 
	$$\rk E_\star > \sqrt{g - \delta - \ell} + \ell.$$
\end{corollary}
\begin{proof}
	We seek to find an upper bound on the expression $S := \sum_{i=0}^{\ell - 1} r_{i+1} r_i$, given that all $r_i \geq 1$ and $\sum_{i=0}^\ell r_i = r$. Applying the AM-GM inequality to each pair of terms $r_i r_{i+1}$, we see that $S \leq \frac{1}{4} \sum_{i=0}^\ell (r_i + r_{i+1})^2.$ By applying the Cauchy-Schwarz inequality to each expression $(r_i + r_{i+1})^2$, we further conclude that $S \leq \frac{1}{2} \sum_{i=0}^\ell r_i^2$. 
	
	To conclude, we note that by convexity of the function $x^2$, the expression $\sum_{i=0}^\ell r_i^2$ will be maximized when all but one of the $r_i$ are equal to one. Indeed, if $r_i \geq r_j$, then adding one to $r_i$ and subtracting one from $r_j$ will change $\sum_{i=0}^\ell r_i^2$ by 
	$$ (r_i+1)^2 + (r_j -1)^2 - (r_i + r_j)^2 = 2(r_i - r_j) + 2 >0.$$
	Thus $\sum_{i=0}^\ell r_i^2 \leq ((\sum_{i=0}^\ell r_i) - \ell)^2 + \ell$. We conclude then that $(r- \ell)^2 + \ell \geq C$.
\end{proof}

With this, we may now prove \Cref{unitarymonodromy}, which aims to show that if an isomonodromic deformation of $(E_\star, \nabla)$ underlies a variation of Hodge structure and if $\rk E_\star < \ssbound$, then $(E_\star, \nabla)$ has unitary monodromy.  With the criterion of semistability provided by \Cref{rankbound}, the proof is exactly the same as \cite{landesman2025geometriclocalsystemsgeneral}[Theorem 1.2.13], but we recapitulate it here.
\begin{proof}[Proof of \Cref{unitarymonodromy}]
	Denote an isomonodromic deformation of $(E_\star, \nabla)$ by $(E_\star', \nabla')$, and denote its Deligne canonical extension by $(\bar{E}_\star', \bar{\nabla}')$. Since $(\bar{E}_\star', \bar{\nabla}')$ is semistable and underlies a variation of Hodge structure, if the Hodge filtration $F^\bullet$ is nontrivial, then any Higgs field $F^i \bar{E}_\star' \to \gr^{i-1}_F \bar{E}_\star' \otimes \omega_C$ must vanish, for otherwise the image would destabilize $\bar{E}_\star'$. Thus the monodromy preserves the polarization on $\bar{E}_\star'$ and the monodromy representation is unitary. By definition of isomonodromic deformations, the same is true for $\bar{E}_\star$. 
\end{proof}

\subsection{Bounds on the Monodromy Group}
In this section, we prove \Cref{liealgbound}. After some setup, the proof is much the same as that of \Cref{parabolicrkbound}. Let $(E_\star, \nabla)$ be a parabolic bundle which underlies a polarizable complex variation of Hodge structure. Let $G$ be the algebraic monodromy group of $(E_\star, \nabla)$, and let $\mathfrak{g}$ be its Lie algebra. 

\begin{proof}[Proof of \Cref{liealgbound}]
	1) By the definition of the isomonodromic deformation of $(E_\star, \nabla)$, the restrictions of the isomonodromic deformation of $(E_\star, \nabla)$ to a very general curve all have the same corresponding monodromy representation, and hence they all have the same algebraic monodromy group. It follows from \Cref{subvhsend} that after replacing $(E_\star, \nabla)$ by its isomonodromic deformation, $\mathfrak{g}$ is a parabolic subbundle of $\End(E_\star)$. In particular, if $D$ is the connection on $\End(E_\star)$, then we have $D(\mathfrak{g})\subset \mathfrak{g} \otimes \Omega_C(D)$. Since $E_\star$ underlies a variation of Hodge structure, $\mathfrak{g}$ has a Hodge filtration, which we denote by $F^\bullet$.
	
	2) By steps 1-3 of the proof of \Cref{parabolicrkbound}, the deformed connection induces a map $\psi: T_C(-D) \to \gr_F^{-1} \End(E_\star)$. The Serre dual of this map is $\psi^\vee: \gr^1_F \End(E_\star) \otimes \Omega_C(D) \to \Omega_C(D)^2.$ This map comes from a connection on $\End(E_\star)$, and as such it preserves $\gr^1_F \mathfrak{g} \otimes \Omega_C(D)$. As such, we have a map $\gr^1_F \mathfrak{g} \otimes \Omega_C(D) \to \Omega_C(D)^2$ which induces a rank-at-most-$\delta$ map on $H^0$. 
	
	3) The map $\gr^1_F \mathfrak{g} \otimes \Omega_C(D) \to \Omega_C(D)^2$ comes from a Higgs field, hence it factors through $Q := \coker(\ad \theta: \gr^2_F \mathfrak{g} \to \gr^1_F \mathfrak{g} \otimes \Omega_C(D))$. By the Grothendieck spectral sequence, there is a surjection $H^0(\gr^1_F \mathfrak{g} \otimes \Omega_C(D)) \to H^0(Q)$, so we get a map $Q \to \Omega_C(D)^2$ inducing a map of rank-at-most-$\delta$ on $H^0.$ By our assumptions and since $\mathfrak{g}$ underlies a complex variation of Hodge structure, $\mu_\star(\gr^1_F \mathfrak{g} \otimes \Omega_C(D)) > 2g-2 + n$. If follows that $\mu_\star (Q) > 2g-2+n$ as well. 
	
	4) We have isolated a quotient bundle of $\mathfrak{g}$ which, numerically, is essentially the same as the one in the proof of \Cref{parabolicrkbound}. As such, we may invoke \Cref{nonggg} in the same way to conclude that $\rk \gr_F^1 \mathfrak{g} > g- \delta$. If $\ell$ is the level of $E_\star$, then the level of $\mathfrak{g}$ is $2k+1 < 2\ell + 1$. By the irreducibility of $\nabla$ we may assume that all the graded pieces of the Hodge filtration of $\mathfrak{g}$ are nonzero, hence we may conclude that $\rk \mathfrak{g} > g+2k -\delta$.
\end{proof}
\subsection{Dimension of Nonabelian Hodge Loci}
    We now specialize to the case of a flat vector bundle that underlies a polarizable $\CC$-VHS, in which case our result implies \Cref{nhldimension}. Recall that we are given a connected component $Z \subset \Hdg$ whose very general point underlies a polarizable complex variation of Hodge structure of level $\ell$. If $\shl \subset \shm_{dR}(\shc/\shm_{g,n}, r)$ is a leaf of the isomonodromy foliation, then we wish to show that the codimension of $Z^{dR} \cap \shl$ is greater than $\codimlowerbound$. 
        \begin{proof}[Proof of \Cref{nhldimension}]
            Under the natural map $\shm_{dR}(\shc/\shm_{g,n}, r) \to \shm_{g,n}$, we may view $Z^{dR} \cap \shl$ as an analytic subvariety of $\shm_{g,n}$ of codimension $\delta$. Hence if $[(C,D)] \in Z^{dR} \cap \shl$, it carries a flat vector bundle $(E, \nabla)$ such that the isomonodromic deformation of the Deligne canonical extension of $(E,\nabla)$ to a very general point of $Z^{dR}\cap \shl$ underlies a polarizable complex variation of Hodge structure of level $\ell$. Thus we may use the conclusion of \Cref{convexarg} and solve for $\delta$ We have 
            \begin{align*}
            	r &> \sqrt{g -\delta -\ell} + \ell
            \end{align*} 
            Whence
            \begin{align*}
            	(r- \ell)^2 &> g -\delta -\ell
            \end{align*}
            Thus
			\begin{align*}
            	\delta > g - (r-\ell)^2 - \ell.
            \end{align*}
			
			This gives the lower bound on $\delta$. 
        \end{proof}

\printbibliography
\end{document}